\documentclass[12pt,a4paper,leqno]{amsart}

\title[]{Microlocal analysis of quasianalytic Gelfand-Shilov type ultradistributions}

\author[M. Cappiello]{Marco Cappiello}

\address{Department of Mathematics, University of Torino, Via Carlo Alberto 10, 10123 Torino, Italy.}

\email{marco.cappiello@unito.it}

\author[R. Schulz]{Ren\'e Schulz}

\address{Mathematisches Institut, Georg-August Universit\"at G\"ottingen, Bunsenstra\ss e 3--5, D--37073 G\"ottingen,
Germany}

\email{rschulz@uni-math.gwdg.de}

\usepackage[pdftex]{graphicx}
\usepackage{layout}
\usepackage{fancyhdr}
\usepackage{setspace}
\usepackage[english]{babel}
\usepackage{amsmath,amssymb,amstext,amsthm,amsfonts,color,ngerman,bbm,calc,graphicx,accents,enumitem}
\usepackage[arrow, matrix, curve]{xy}
\usepackage[ansinew]{inputenc}

\numberwithin{equation}{section}          

\newtheorem{thm}{Theorem}
\numberwithin{thm}{section}

\newcommand{\rubrik}{}
\newtheorem{prop}[thm]{Proposition}
\newtheorem{cor}[thm]{Corollary}
\newtheorem{lem}[thm]{Lemma}

\theoremstyle{definition}

\newtheorem{defn}[thm]{Definition}

\theoremstyle{remark}

\newtheorem{rem}[thm]{Remark}

\newcommand{\thmref}[1]{Theorem~\ref{#1}}

\newcommand{\lemref}[1]{Lemma~\ref{#1}}

\newcommand{\RR}{{\mathbb R^d}}
\newcommand{\NN}{\mathbb N^d}

\newcommand{\no}{\mathbb N}
\newcommand{\rr}[1]{\mathbb R^{#1}}
\newcommand{\GS}[1]{\mathcal{S}_{#1}}
\newcommand{\GSd}[1]{{\mathcal{S}_{#1}^\prime}}
\newcommand{\GSdG}[1]{{\mathcal{S}_{#1 ,\Gamma}^\prime}}
\newcommand{\nn}[1]{\mathbb N^{#1}}

\newcommand{\CC}{{\mathbb C^d}}
\newcommand{\co}{\mathbb C}
\newcommand{\cc}[1]{\mathbb C^{#1}}

\newcommand{\WF}{\mathrm{WF}_\mathrm{gl}}
\newcommand{\WFt}{\mathrm{WF}^\theta_\mathrm{gl}}

\newcommand{\wt}{\widetilde}

\def\N{\mathbb{N}}

\keywords{Gelfand-Shilov spaces, global wave front sets, Bargmann transform, localization operators, tempered ultradistributions. }

\begin{document}

\begin{abstract}
We introduce a global wave front set suitable for the analysis of tempered ultradistributions of quasianalytic Gelfand-Shilov type. We study the transformation properties of the wave front set and use them to give microlocal existence results for pull-backs and products. We further study quasianalytic microlocality for classes of localization and ultradifferential operators, and prove microellipticity for differential operators with polynomial coefficients.
\end{abstract}

\maketitle

\section{Introduction}

Starting from \cite{Hormander0}, the analysis of singularities of Schwartz distributions has been based on the study of their wave front set. We 
recall that a distribution $u \in \mathcal{D}'(\RR)$ is said to be microlocal at a point $(x_0, \xi_0) \in \RR \times (\RR \setminus 0)$ if there exists a cut-off function $\phi$ supported around $x_0$ such that $\widehat{\phi u}$ is rapidly decaying in a conic neighborhood of $\xi_0.$ The wave front set is then defined as the set of points $(x,\xi) \in \RR \times (\RR \setminus 0)$  where $u$ is not microlocal. The approach introduced by H\"ormander has been applied successfully to the analysis of propagation of singularities in the theory of partial differential equations. Moreover, it has been soon extended to the analysis of other types of singularities such as Gevrey singularities, cf. \cite{Rodino} and the references therein, or analytic singularities, see e.g. \cite{Hormander0, SJ}. 
\par
Later many authors started to study \textit{global} singularities of tempered distributions in $\mathcal{S}'(\RR)$ and introduced several different notions of wave front sets providing information not only on the local regularity of the distributions but also on their behavior at infinity. Among them we recall the \textit{scattering wave front set} $\textrm{WF}_{\textrm{sc}}$, see \cite{Melrose1}, also known as $\mathcal{S}$-\textit{wave front set} $\textrm{WF}_{\mathcal{S}}$, see \cite{Cordes,CM}, and the \textit{quadratic scattering wave front set} $\textrm{WF}_{\textrm{qsc}}$, see \cite{Wunsch}, which appear as the natural tools to study the propagation of singularities on manifolds with conical ends. Also the analytic counterpart of these wave front sets has been defined in \cite{RZ1, RZ2}.
\par
Another notion of global wave front set for tempered distributions has been introduced in \cite{Hormander1}. In this context a distribution $u \in \mathcal{S}'(\RR)$ is microlocal at a point $(x_0,\xi_0) \in T^{\ast}(\RR) \setminus \{(0,0)\}$ if there exists a Shubin type symbol $a(x,\xi)$ such that $a^w(x,D)u \in \mathcal{S}(\RR)$ and $a$ is non-characteristic at $(x_0,\xi_0),$ (cf. also \cite{RW} for a different characterization of this set in terms of the Gabor transform). More recently, an equivalent notion of wave front set, called \textit{homogeneous wave front set}, has been defined in \cite{Nakamura1} in the language of semi-classical analysis. The equivalence between the two wave front sets has been indeed proved only very recently in \cite{SW}. The homogeneous wave front set, as well as its Gevrey and analytic versions, see \cite{MNS1, MNS2, Mizuhara}, has been mainly applied to study the microlocal smoothing effect for Schr\"odinger equations.
\par
In this paper we are interested in the case of tempered ultradistributions of Gelfand-Shilov type. The Gelfand-Shilov space $\mathcal{S}_\theta (\RR), \theta \geq 1/2,$ has been defined in \cite{GS} as the space of all functions $f \in C^{\infty}(\RR)$ satisfying the following estimates
\begin{equation} \label{GSequal}
|\partial^{\alpha}f(x)| \leq C^{|\alpha|+1}\alpha!^{\theta}e^{-c|x|^{\frac{1}{\theta}}}, \qquad x \in \RR,
\end{equation}
for some positive constant independent of $\alpha.$ Clearly we have $\mathcal{S}_\theta (\RR) \subset \mathcal{S}(\RR).$ We refer to \cite{GS, Pilipovic} and to the next Section 1.1 for a more extended description of the properties of this space and of its dual space $\mathcal{S}_\theta^{\prime}(\RR)$. Here we limit ourselves to observe that with respect to the definition of $\mathcal{S}(\RR)$, the estimate \eqref{GSequal} gives detailed information on the Gevrey-analytic regularity of $f$ and on the exponential decay at infinity. For $\theta >1,$ these spaces appear as a natural global counterpart of the Gevrey spaces, whereas for $\theta=1$ their elements are real-analytic functions of exponential decay. Concerning microlocal analysis, a notion of scattering wave front set has been introduced in \cite{CR} by a quite easy translation of the approach used in \cite{CM} for tempered distributions to the Gelfand-Shilov frame. Similarly, the extensions of the homogeneous wave front set in \cite{MNS1, MNS2, Mizuhara}, encoding Gevrey and analytic singularities of tempered distributions, can be easily extended to the elements of the dual space of $\mathcal{S}_\theta (\RR),$ $\theta \geq 1.$ 
\par
For $\theta < 1,$ the scattering wave front set cannot be so easily converted, as the space $\mathcal{S}_\theta (\RR)$ lacks compactly supported functions. On the other hand, these spaces are very interesting for their connections with complex analysis as the space of test functions consists in entire functions and they are a natural framework for studying regularity properties of solutions of several classes of partial differential equations, see \cite{CGR1, CGR2, Hormander1}. 
Moreover, they are currently studied also in other fields as for instance time-frequency analysis due to their good invariance properties with respect to the short time Fourier transform, cf. \cite{CPRT, Grochenig2, Teofanov, Toft}.
\par
A first approach to microlocal analysis in these spaces would be to approximate cut-off functions as done in the local theory in \cite{Hormander0}.
However, this method is technically very involved and yields several different notions for the case $\theta <1,$ see \cite{PT}.\\ 
In our paper, we adopt the approach already outlined in \cite{Hormander1} for the case $\theta =1/2$, and based on the characterization of singularities via the so-called Fourier-Bros-Iagolnitzer (FBI) transform, also known as Bargmann transform. This approach, which has been already used in the local theory, cf. \cite{CK, Hormander0, Martinez1, SJ}, is more appropriate for our functional setting. Moreover, due to the relation between the FBI transform and the short time Fourier transform, we can base the proofs of some statements on well-known results from time-frequency analysis and we are also able to study the properties of our wave front sets under the action of localization operators. This is an entirely new feature for this wave front set, as previously microlocal properties of localization operators were only known through their connection to Weyl pseudodifferential operators.\\
Our wave front set is modelled similarly to the homogeneous wave front set in \cite{MNS1, MNS2, Mizuhara, Nakamura1}, in particular it is defined as a conic set in $\mathbb{R}^{2d} \setminus \{(0,0)\}.$ Here we focus mainly on the general properties of this wave front set and postpone to a future paper possible applications to particular classes of partial differential equations.

The paper is organized as follows.  In Section \ref{preliminaries} we recall the definition and some basic facts on the Gelfand-Shilov spaces and their duals and we prove some mapping properties of the FBI transform on these spaces. In Section \ref{wavefront} we define the global wave front set for tempered ultradistributions and illustrate its behavior under linear symplectic transformations and standard operations such as pullbacks, tensor products, products and convolution. In Section \ref{micro} we study the microlocality properties of the wave front set with respect to localization operators and to differential and ultradifferential operators. Finally we prove a microellipticity result for differential operators with polynomial coefficients in analogy with what has been done in \cite{Hormander1} in the case $\theta =1/2.$
\section{Preliminaries}\label{preliminaries}

In the sequel we will use the notation $A(x)\lesssim B(x)$ if two maps $A$ and $B$ from some set $X$ to $[0,+\infty)$ fulfill $A(x)\leq CB(x)\ \forall x \in X$ for some positive constant $C$ independent of all possible indices involved. 

\subsection{Gelfand-Shilov functions and ultradistributions}

In the following let $\mu, \nu> 0$ with $\mu+\nu\geq 1$.

\begin{defn}
Let $C>0$. We denote by $\mathcal{S}^{\mu}_{\nu,C}(\RR)$ the Banach space of all $\mathcal{C}^\infty$-functions satisfying
	\begin{equation}
\sup_{\alpha, \beta \in \nn{d}} \sup_{x \in \RR} C^{-|\alpha|-|\beta|}(\alpha!)^{-\nu}(\beta!)^{-\mu}|x^\alpha\partial^\beta f(x)|<\infty,
		\label{eq:gsdef}
	\end{equation}
equipped with the norm given by the lefthand side.\\
The space of all Gelfand-Shilov functions of indices $\mu,\nu$ on $\RR$ is then defined as 
$$\mathcal{S}^{\mu}_{\nu}(\RR):=\bigcup_{C>0} \mathcal{S}^{\mu}_{\nu,C}(\RR)$$
equipped with the inductive limit topology. 
\end{defn}
There are other equivalent ways of defining the space $\mathcal{S}^{\mu}_{\nu}(\RR)$, cf. \cite{CCK, GS}. We list some of them in the following lemma.
\begin{lem}Let $\mu > 0, \nu > 0, \mu+\nu \geq 1.$ For $f \in \mathcal{S}(\RR)$ the following conditions are equivalent:\\
i) $f \in \mathcal{S}^{\mu}_{\nu}(\RR).$ \\
ii) There exist positive constants $C, c$ such that $$|\partial^{\alpha}f(x)|\lesssim C^{|\alpha|} (\alpha!)^{\mu}e^{-c|x|^{1/\nu}}, \qquad x \in \RR.$$
iii) There exist positive constants $A,B$ such that
$$\sup_{x \in \RR}|x^{\beta}f(x)| \lesssim A^{|\beta|}(\beta!)^{\nu} \quad \textrm{and} \quad \sup_{x \in \RR}|\partial^{\alpha}f(x)| \lesssim B^{|\alpha|}(\alpha!)^{\mu}.$$\\
iv) There exists $\varepsilon>0$ such that
$$|f(x)| \lesssim e^{-\varepsilon|x|^{1/\nu}}, \quad x \in \RR, \quad \textrm{and} \quad |\hat{f}(\xi)| \lesssim e^{-\varepsilon|\xi|^{1/\mu}}, \quad \xi \in \RR.$$
\end{lem}
In this paper, we are concerned with the quasianalytic case, i.e. $\mu<1$. In these spaces, we have additional properties concerning holomorphic extensions:
\begin{prop}
If $f \in \mathcal{S}^{\mu}_{\nu}(\RR), 0<\mu<1, \nu >0,$ then $f$ extends to an
 entire analytic function $f(x+iy)$ in $\mathbb{C}^d$, which satisfies the following estimate:
$$|f(x+iy)|\lesssim e^{-\varepsilon|x|^{\frac{1}{\nu}} + \delta|y|^{\frac{1}{1-\mu}}}, \quad x \in \RR, y \in \RR,$$
for some $\delta,\epsilon>0$. 
\end{prop}

\begin{defn}
We denote by $(\mathcal{S}^{\mu}_{\nu})'(\RR)$ the topological dual of $\mathcal{S}^{\mu}_{\nu}(\RR)$ (endowed with the weak topology).\\
More precisely: a linear form $u:\mathcal{S}^{\mu}_{\nu}(\RR)\rightarrow \mathbb{C}$ is in $(\mathcal{S}^{\mu}_{\nu})'(\RR)$ if for every $\varepsilon>0$ there exists a positive constant $C_{\varepsilon}$ such that
	\begin{equation}
		|u(f)|\lesssim C_{\varepsilon}\sup_{\alpha,\beta \in \no^d} \varepsilon^{-|\alpha|-|\beta|} (\alpha!)^{-\nu}(\beta!)^{-\mu}\sup_{x\in\RR}|x^\alpha\partial^\beta f(x)|
	\end{equation}
	for every $ f \in S^{\mu}_{\nu}(\RR).$
\end{defn}

\subsection{Integral transforms on Gelfand-Shilov spaces}

\subsubsection{Fourier transform}
Already in \cite{GS} it was shown that the Gelfand-Shilov spaces are invariant under translations and dilations and that they have the following behaviour under Fourier transforms:

\begin{lem}
The Fourier transformation maps, continuously, $\mathcal{S}^{\mu}_{\nu}(\RR)$ into $\mathcal{S}^{\nu}_{\mu}(\RR)$ and $(\mathcal{S}^{\mu}_{\nu})'(\RR)$ into $(\mathcal{S}^{\nu}_{\mu})'(\RR)$.
\end{lem}

To have full metaplectic invariance, in the sequel we shall restrict ourselves to the case of $\mu=\nu=:\theta$ with $\frac{1}{2} \leq \theta<1$ and we shall write $\mathcal{S}_\theta(\RR), \mathcal{S}_\theta^\prime(\RR)$ to denote the space $\mathcal{S}^\theta_\theta(\RR)$ and its dual. \\

There are several other transforms that have good properties with respect to Gelfand-Shilov spaces. One of them is the following Bargmann- or FBI-transform, stemming from \cite{Hormander1}:

\subsubsection{FBI-transform with general phases}

\begin{defn} \label{definitionphase}
Let $\varphi$ be a quadratic form in $\CC\times\RR$, i.e.
$\varphi(z,y)=\frac{1}{2}\langle Az,z\rangle+\langle Bz,y\rangle+\frac{1}{2}\langle Cy,y\rangle$ with
 $A$, $B$ and $C$ complex $d\times d$-matrices such that $A$ and $C$ are symmetric.  
Assume further non-degeneracy conditions, i.e. $\det B=\det \varphi^{\prime\prime}_{zy}\neq0$, and $\Im C=\Im \varphi^{\prime\prime}_{yy}>0$.\\
For $u\in\GSd{\theta}(\RR)$ we denote by 
	\begin{equation}
	\label{eq:fbidef2}
		T_\varphi(u)(z)=c_\varphi \langle u_y, e^{i\varphi(z,y)}\rangle,
	\end{equation}
the FBI transform with phase $\varphi$ of $u$. Herein, $$c_\varphi:=2^{-\frac{d}{2}}\pi^{-\frac{3d}{4}}(\det \Im C)^{-\frac{1}{4}}|\det B|$$
 and the pairing is with respect to the $y$-variable.
\end{defn}
In the sequel we shall denote $\Phi(z)= \max (-\Im \varphi(y,x))$ and by $y(z)$ the point where such a maximum is attained. We then observe that $y(z)$ is a linear function of $z$ and that $$-\Im \varphi(y,z) -\Phi(z)=-\frac{1}{2}\langle (\Im C)(y-y(z)), y-y(z) \rangle.$$
\begin{rem}
\label{rem:stdphase}
The most interesting case is that of the phase $$\phi(z,y):=\langle z,y\rangle + \frac{i}{2} |y|^2 ,$$ for which we have $c_\phi:=2^{-\frac{d}{2}}\pi^{-\frac{3d}{4}}$. In this case, $T_{\phi}$ turns out to coincide with the inverse Fourier-Laplace transform of
$v(y)=e^{-\frac{1}{2}|y|^{2}}u(y).$ Moreover, in this case we have $\Phi(z)=\frac{1}{2}|\Im(z)|^2$.
\end{rem}
We now study the properties of the Bargmann transform of tempered ultradistributions. Let us first recall a classic result, which will be useful in the sequel: The Bargmann transform takes $L^2$-functions to entire functions that are $L^2$ with respect to a Gaussian measure, the so-called Bargmann-Fock space. A proof in our context can be found in \cite{Hormander1}.
\begin{lem} \label{L2cont}
Let $\varphi$ be as in Definition \ref{definitionphase}. Then 
\begin{equation}\label{isometry}\int |T_{\varphi}u(z)|^2 e^{-2\Phi(z)}d\lambda (z) = \|u \|_{L^2(\RR)}. \end{equation}
where $d\lambda(z)$ denotes the Lebesgue measure on $\mathbb{C}^d$. In other words, the map $u \to T_{\varphi}u$ is an isometry from $L^2(\RR)$ into $L^2(\mathbb{C}^d, d\lambda(z) e^{-2\Phi(z)}) \cap \mathcal{H}(\mathbb{C}^d)$, where $\mathcal{H}(\mathbb{C}^d)$ denotes the space of entire functions on $\CC.$
\end{lem}
The following statement is an analogous version with respect to Gelfand-Shilov functions and tempered ultradistributions. We include a proof for self-containedness, but also refer to \cite{Grochenig2, Teofanov, Toft} for similar statements.
\begin{thm}
\label{thm:wfprop}
Let $u\in \GSd{\theta}(\RR)$ and let $\varphi$ be a FBI-phase. Then $T_\varphi u$ is an entire analytic function which satisfies for any $\alpha \in \no^d$:
\begin{align} \label{eq:gsdstf}
u\in\GSd{\theta}(\RR)&\Rightarrow \forall \varepsilon>0\ \exists C_\varepsilon>0\text{ s.t.} &|\partial_z^\alpha T_\varphi u(z)|&\leq  C_\varepsilon^{|\alpha|+1}\alpha!^\theta\ e^{\varepsilon |z|^{1/\theta}+\Phi(z)} \\ \label{eq:gsstf}
u\in\GS{\theta}(\RR)&\Rightarrow \exists \varepsilon>0 , C>0 \text{ s.t.} &|\partial_z^\alpha T_\varphi u(z)|&\leq C^{|\alpha|+1}\alpha!^\theta \ e^{-\varepsilon |z|^{1/\theta}+\Phi(z)}
\end{align}
Conversely, if there is an analytic function $U$ satisfying the estimate in \eqref{eq:gsdstf} (respectively \eqref{eq:gsstf}), then there exists a unique $u\in \GSd{\theta}(\RR)$ (respectively $u\in\GS{\theta}(\RR)$) s.t. $U=T_\varphi u$.\\
%
\end{thm}
We can reduce the proof to the case when $\alpha=0$ thanks to the following Lemma, a variant of Cauchy's inequality.
\begin{lem}
\label{lem:cauchyineq}
Let $g(z)$ be an entire function and let $z_0\in\mathbb{C}^d \setminus \{0\}.$ Let open conic subsets of $\mathbb{C}^d \setminus \{0\}$ containing $z_0$ be denoted by $\Gamma$. Then the following conditions are equivalent:
\\
i) $\exists\Gamma\ \forall \varepsilon>0\ \exists C_\varepsilon>0\text{ s.t.} \left| g(z)\right|\leq  C_\varepsilon e^{\varepsilon |z|^{1/\theta}+\Phi(z)}\quad \forall\ z\in\Gamma$\\
ii) $\exists\Gamma\ \forall \varepsilon>0\ \exists C_\varepsilon>0\text{ s.t.} \left|\partial_z^\alpha g(z)\right|\leq  C_\varepsilon^{|\alpha|+1} \alpha!^\theta \ e^{\varepsilon |z|^{1/\theta}+\Phi(z)}\quad \forall\ z\in\Gamma,\alpha\in\NN.$\\
Similarly, the following two conditions are equivalent:\\
a) $\exists\Gamma\ \exists \varepsilon>0\ \exists C>0\text{ s.t.}\left| g(z)\right|\leq  C e^{-\varepsilon |z|^{1/\theta}+\Phi(z)}\quad \forall\ z\in\Gamma$\\
b) $\exists\Gamma\ \exists \varepsilon>0\ \exists C>0\text{ s.t.} \left|\partial_z^\alpha g(z)\right|\leq  C^{|\alpha|+1} \alpha!^\theta \ e^{-\varepsilon |z|^{1/\theta}+\Phi(z)}\quad \forall\ z\in\Gamma,\alpha\in\NN.$
\end{lem}
\begin{proof}
The proof is a variant of a calculation done in \cite[Section 5.2]{GS}.
Let $g(z)$ satisfy i) on an open cone $\Gamma$ and let $z\in \Gamma.$  Then, by Cauchy's formula we can write
$$\partial_z^\alpha g(z)= \frac{\alpha!}{(2\pi i)^d}\int_{\partial \mathbb{B}_\delta(z)} g(a) \prod_{i=1}^d (a_i-z_i)^{-1-\alpha_i}\ da,$$
where we denote by $\mathbb{B}_\delta(z)$ the polydisc of radius $\delta$ centered in $z.$ For fixed $\delta$ we can assume $|z|$ large enough 
so that $\mathbb{B}_\delta(z) \subset \Gamma.$ As $\Phi(z)$ is a polynomial in $z,\ \overline{z}$ of degree $2$ we can estimate it on $\partial\mathbb{B}_\delta(z)$ by Taylor's formula and Young's inequality for products:
$$\Phi(a)\leq \Phi(z)+C_1 |z|\delta + C_2 \delta^2\leq \Phi(z)+\varepsilon |z|^{1/\theta} + c_\epsilon \delta^{1/(1-\theta)},$$
for any $\epsilon>0$. Then, using i) we can estimate as follows
\begin{align*}
|\partial_z^\alpha g(z)| &\leq \alpha!\sup_{a \in \partial \mathbb{B}_\delta(z)}|g(a)|   \\
&\leq \sup_{a\in \partial \mathbb{B}_\delta(z)} \frac{\alpha!\ C_{\varepsilon}}{|\delta|^{|\alpha|}}  e^{\varepsilon |a|^{1/\theta}+\Phi(a)} \\
&\leq \frac{\alpha!\ C_{\varepsilon^\prime}}{|\delta|^{|\alpha|}}  e^{\varepsilon^\prime |z|^{1/\theta}+\Phi(z)} e^{c_\varepsilon^\prime \delta^{1/(1-\theta)}}\\
\end{align*}
We now pick $\delta=\delta_0$ to minimize $ \frac{\alpha!\ C_{\varepsilon^\prime}}{|\delta|^{|\alpha|}} e^{c_\varepsilon^\prime \delta^{1/(1-\theta)}}:=C(\varepsilon,\alpha,\delta).$
We have $\delta_0=\left(\frac{|\alpha|(1-\theta)}{c_\varepsilon^\prime}\right)^{1-\theta}$ and using Stirling's formula we obtain $$C(\varepsilon,\alpha,\delta_0)\leq \left(C_\epsilon^{\prime\prime}\right)^{|\alpha|+1}\alpha!^\theta,$$ which proves estimate ii).\\
The second part of the lemma follows in complete analogy. 
\end{proof}
\begin{proof}[Proof of \thmref{thm:wfprop}]
Following \cite{Hormander1}, we can assume $A=0$ and $\Re C =0$ in the definition of $\varphi$. By a change of notation this leads to consider
a phase function of the form
$$\varphi(y,z)= \langle Bz,y\rangle +\frac{i}{2}\langle Cy,y \rangle,$$
where $B$ is a non degenerate matrix and $C$ is positive definite. Let now $u \in \GSd{\theta }(\RR).$ We have that for every $\varepsilon >0$ there exists $C_{\varepsilon}>0$ such that
$$|T_{\varphi}u(z)|= c_\varphi |\langle u_y, e^{i\varphi(\cdot, z)}\rangle | \leq C_{\varepsilon} \sup_{\alpha, \beta \in \N^d}\varepsilon^{-|\alpha|-|\beta|}(\alpha!\beta!)^{-\theta} \sup_{y \in \RR} |y^{\beta}D_y^{\alpha}e^{i\varphi(y,z)}|.$$
Now observe that $e^{-\frac{1}{2}\langle Cy,y\rangle }\in \mathcal{S}_{1/2}(\RR)$ and that $$|D_y^\alpha e^{-\frac{1}{2}\langle Cy,y\rangle } |\leq M^{|\alpha|+1}\alpha!^{1/2}e^{-\frac{1}{2}\langle Cy,y\rangle }$$ for some positive constant $M$. Then we have
\begin{multline*}
|y^{\beta}D_y^{\alpha}e^{i\varphi(y,z)}| \leq |y|^{|\beta|} \sum_{\gamma \leq \alpha } \binom{\alpha}{\gamma} |(Bz)^{\gamma} e^{i\langle Bz,y\rangle}D_y^{\alpha-\gamma}e^{-\frac{1}{2}\langle Cy,y\rangle } | \\ 
\leq C_1^{|\alpha|+1} \sum_{\gamma \leq \alpha } \binom{\alpha}{\gamma}|y|^{|\beta|}|z|^{|\gamma|}(\alpha-\gamma)!^{1/2}e^{\Phi(z)-\frac{1}{2}\langle  C(y-y(z)), y-y(z) \rangle }.
\end{multline*}
Since $\theta \geq 1/2$ and $y(z)$ is a linear function of $z$ we have that for every $\varepsilon'>0$ there exists $C_{\varepsilon'}>0$ such that
\begin{multline*}
|z|^{|\gamma|}\sup_{y \in \RR} |y|^{|\beta|}e^{-\frac{1}{2}\langle C(y-y(z)), y-y(z) \rangle} = |z|^{|\gamma|}\sup_{y \in \RR} (|y+y(z)|)^{|\beta|}e^{-\frac{1}{2}\langle Cy, y \rangle} \\ \leq C\sup_{y \in \RR} (|y|+|z|)^{|\beta|+|\gamma|}e^{-\frac{1}{2}\langle Cy, y \rangle} \leq C_{\varepsilon'}^{|\beta|+|\gamma|+1}(\beta!\gamma!)^{\theta}e^{\varepsilon' |z|^{1/\theta}}. 
\end{multline*}
Then, taking $\varepsilon > C_{\varepsilon'}C$, we obtain \eqref{eq:gsdstf}. \\ Let now $u \in \GS{\theta}(\RR).$ We observe that for every $\alpha \in \N^d$:
$$(Bz)^{\alpha}T_{\varphi}u(z) =(-1)^{|\alpha|} \int_\RR e^{i\langle Bz,y \rangle}D_y^{\alpha}
\left( e^{-\frac{1}{2}\langle Cy,y \rangle}u(y)\right) \, dy. $$
Then, since $\textrm{det}B \neq 0$ and $\theta \geq 1/2$, we have by Leibniz and Fa\`a di Bruno formulas
$$
|z^{\alpha}T_{\varphi}u(z)| \leq C_1^{|\alpha|+1}(\alpha!)^{\theta}e^{\Phi(z)} \int_\RR e^{-\frac{1}{2}\langle C(y-y(z)), y-y(z)\rangle }\, dy 
$$
 for every $\alpha \in \no^d.$ This gives \eqref{eq:gsstf}. \\ To prove the second part of the Proposition, we need Lemma \ref{L2cont}. Let $U(z)$ be an entire function satisfying \eqref{eq:gsstf} and let $u = T^{\ast}_{\varphi}U,$ where 
$T^{\ast}_{\varphi}$ is the $L^2$-adjoint of $T_{\varphi}.$ Then $T_{\varphi}u=U$. Moreover we have 
$$T^{\ast}_{\varphi}U(y)= c_{\varphi} \int_\CC e^{-i \overline{\varphi(y,z)}-2\Phi(z)}U(z)d\lambda (z) \in \GS{\theta}(\RR).$$
As a matter of fact, given $\alpha, \beta \in \N^d$ and arguing as in the proof of \eqref{eq:gsdstf}, we have:
\begin{multline*}
|y^{\beta}D_y^{\alpha}T^{\ast}_{\varphi}U(y)| = c_{\varphi}\left|y^{\beta} \int_\CC D_y^{\alpha} e^{-i \overline{\varphi(y,z)}}e^{-2\Phi(z)}U(z)d\lambda(z) \right| \\
\leq C^{|\alpha|} \sum_{\gamma \leq \alpha} \binom{\alpha}{\gamma}(\alpha-\gamma)!^{1/2} \int_\CC |y|^{|\beta|}|z|^{|\gamma|}\left|e^{-i\overline{\varphi(y,z)}}e^{-2\Phi(z)}U(z) \right| d\lambda(z) \\ 
\leq C_{\varepsilon}C^{|\alpha|} \sum_{\gamma \leq \alpha} \frac{\alpha!}{\gamma!(\alpha-\gamma)!^{1/2}} \int_\CC |y|^{|\beta|}|z|^{|\gamma|} e^{-\frac{1}{2}\langle C(y-y(z)), y-y(z)\rangle} e^{-\varepsilon |z|^{1/\theta}}d\lambda(z). 
\end{multline*}
Arguing as before we have that 
$$\sup_{y \in \RR}(|y|+|z|)^{|\beta|+|\gamma|} e^{-\frac{1}{2}\langle C(y-y(z)), y-y(z)\rangle} e^{-\varepsilon |z|^{1/\theta}} \leq C_{\varepsilon}C^{|\alpha|+|\gamma|}(\beta!\gamma!)^\theta e^{-\frac{\varepsilon}{2}|z|^{1/\theta}}$$ 
Then, we obtain that $u \in \GS{\theta}(\RR).$ Finally, if $U \in \mathcal{H}(\mathbb{C}^d)$ satisfies \eqref{eq:gsdstf}, then
$$(u,v)= \int_\CC U(z) \overline{T_{\varphi}v(z)}e^{-2\Phi(z)}d\lambda(z), \qquad v \in \GS{\theta}(\RR),$$
defines an ultradistribution in $
\GSd{\theta}(\RR).$ Moreover we have $T_{\varphi}u=U,$ cf. \cite{Hormander1}.
\end{proof}
\begin{rem}
In the case $\phi(z,y)=\langle z,y\rangle + \frac{i}{2} |y|^2$, when $u \in \GS{\theta}(\RR)$, the estimate \eqref{eq:gsstf} takes the form $$|T_{\phi}u(z)| \leq  C e^{-\varepsilon |z|^{1/\theta}+\frac{1}{2}|\Im z|^{2}}, \qquad z \in \mathbb{C}^{d},$$
for some positive constants $C, \varepsilon.$ 
\end{rem} 
We note that in the proof of Theorem \ref{thm:wfprop} we use the following identity, which can be seen as an inversion formula:
\begin{lem}
\label{lem:moyal}
Let $u\in\GSd{\theta}(\RR)$, $f\in\GS{\theta}(\RR)$. Then 
\begin{equation}
\label{eq:moyal}
\langle u,\,f\rangle=\langle T_\varphi u,\, e^{-2\Phi(\cdot)} T_{-\overline{\varphi}} f\rangle =\int_{\CC} T_\varphi u(z)  T_{-\overline{\varphi}} f(z) e^{-2\Phi(z)}d\lambda(z).
\end{equation} 
\end{lem}
\subsection{Short-time Fourier transform}

The Bargmann or FBI transform is deeply connected to another transform, the \emph{short time Fourier transform} (short: \emph{STFT}). For a broad analysis of this connection, we refer to \cite{Grochenig1,Teofanov,Toft} in the setting of Gelfand-Shilov and modulation spaces. In the STFT, we allow more general window functions than Gaussians. Therefore the holomorphicity properties of the transform are less prominent.

\begin{defn}
Let $g\in\GS{\theta}(\RR)\setminus\{0\}$, the so-called \emph{window function}. Then for $u\in\GSd{\theta}(\RR)$ the short time Fourier transform $V_g(u)$, is defined as 
$$V_g(u)(z)=\underbrace{\frac{1}{(2\pi)^{d/2}\|g\|_2}}_{=:k_g}\langle u, \overline{g(\cdot-x)}e^{-i\langle \xi,\cdot\rangle}\rangle \quad \text{with }z=(x,\xi).$$
In particular, $V_\psi(u)$ denotes the transform with the standard Gaussian window $\psi(y):=\pi^{-d/4} e^{-\frac{y^2}{2}}.$ 
\end{defn}

\begin{lem}[Properties of the STFT]
\label{lem:stftproperties}
Let $h\in\GS{\theta}(\RR)$. Then (see \cite{Teofanov}) $V_h$ is a continuous mapping from $\GS{\theta}(\RR)$ to $\GS{\theta}(\rr{2d})$ and $\GSd{\theta}(\RR)$ to $\GSd{\theta}(\rr{2d})$ satisfying $V_h(u)\in\GS{\theta}(\rr{2d})\Rightarrow u\in\GS{\theta}(\RR)$.\\
For the standard Gaussian window $\psi$ and the standard phase $\phi$ we have the following identity:
\begin{equation}
\label{eq:stftrel}
e^{-\Phi(\chi^b_\phi(x+i\xi))}T_\phi(u)(\chi^b_\phi(x+i\xi))=e^{-\frac{|x|^2}{2}}T_\phi(u)(-\xi-ix)=V_\psi(u)(x,\xi).
\end{equation}
Furthermore we have the following estimate for changing to another window $g\in\GS{\theta}(\RR)$ (see \cite[Lemma 11.3.3]{Grochenig1}):
\begin{equation}
\label{eq:groch}
|V_gu |\lesssim \|h\|_2^{-1} \left(| V_hg |*| V_h u |\right).
\end{equation}
\end{lem}

\section{The global wave front set on Gelfand-Shilov spaces}\label{wavefront}

In this section we define the global wave front set for tempered ultradistributions from $\GSd{\theta}(\RR).$ We shall define it, for clearness sake, first with respect to a general FBI phase. First of all let us consider the linear canonical transformation associated to a phase function $\varphi$ and defined as follows:
$$\chi_\varphi : (y, -\varphi'_{y}(y,z)) \to (z, \varphi'_z(y,z)),$$ and let $\chi^b_{\varphi}= \chi_{\varphi} \circ \pi_1,$ where $\pi_1: T^{\ast}\mathbb{C}^d \to \mathbb{C}^d$ is the standard projection. It is known from \cite{SJ} that $\chi^b_{\varphi}$ is a bijection of $T^{\ast}\RR$ on $\mathbb{C}^d.$ 

\begin{defn}
Let $(x_0,\xi_0)\in\rr{2d}\setminus\{0\}=T^*\RR \setminus\{0\}$ and $u\in\GSd{\theta}(\RR)$. We say that $(x_0,\xi_0) \notin\WFt(u)$ if there exists an open conic neighbourhood $\Gamma$ of $\chi^b_{\varphi}(x_0,\xi_0)$ and $\varepsilon>0$ such that 
\begin{equation}
|T_\varphi u(z)|\lesssim \ e^{-\varepsilon |z|^{1/\theta}+\Phi(z)}
\label{eq:decayeq}
\end{equation}
 holds for every $z \in \Gamma$.
\end{defn}
As a consequence of \thmref{thm:wfprop} we get the following result.
\begin{cor}\label{cor:wfempty}
Let $u\in\GSd{\theta} (\RR)$. Then $\WFt(u)=\emptyset$ iff $u\in\GS{\theta} (\RR)$.
\end{cor}
\begin{proof}
If $u \in \GS{\theta} (\RR)$, then $\WFt(u)=\emptyset$ since \eqref{eq:decayeq} holds on all $\mathbb{C}^d.$ Viceversa, if $\WFt(u)=\emptyset$, by the invertibility of $\chi_{\varphi}^b,$ we have that for every $ z \in \mathbb{C}^d$ with $|z| =1$ there exists a conic neighbourhood of $z$ where the estimate in \eqref{eq:gsstf} holds. Since the set $\{z \in \mathbb{C}^d : |z|=1\}$ is compact, then we can find $\varepsilon >0, C>0$ such that
\begin{equation} \label{eq:Test}
|T_{\varphi}u(z)| \leq C e^{-\varepsilon|z|^{1/\theta}+\Phi(z)}
\end{equation}
for every $z \in \mathbb{C}^d$ with $|z| \geq 1.$ For $|z| <1$ the estimate is obviously satisfied. Then \eqref{eq:Test} holds on $\mathbb{C}^d$ and this yields $u \in \GS{\theta} (\RR)$ by \thmref{thm:wfprop}.
\end{proof}
Arguing as in \cite{Hormander1} it is easy to verify that the set $\WFt(u)$ is independent of the choice of $\varphi$. In particular for $\phi(y,z)= \langle z,y\rangle +\frac{i}{2}|y|^2$ we have $\chi^b_{\phi}(x,\xi)= -\xi-i x$, which gives that a point $(x_0,\xi_0) \notin \WFt (u)$ if and only if there exist a conic neighbourhood $\Gamma'$ of $(x_0,\xi_0)$ and $\varepsilon >0$ such that 
\begin{equation} \label{defeq}
\left| \langle u_y,  e^{-i\langle y,\xi \rangle -\frac{1}{2}|y-x|^2} \rangle \right| \leq C e^{-\varepsilon (|x|^{1/\theta}+|\xi|^{1/\theta})} \, \qquad (x,\xi) \in \Gamma'.
\end{equation}
Equivalently, using \eqref{eq:stftrel}, this can be written in terms of the STFT:
\begin{equation} \label{eq:stftdefeq}
\left| V_\psi(u)(x,\xi) \right| \leq C e^{-\varepsilon (|x|^{1/\theta}+|\xi|^{1/\theta})} \, \qquad (x,\xi) \in \Gamma'.
\end{equation}
\begin{lem}\label{lem:convest}
Let $z_0\in\rr{2d}\setminus\{0\}$, $\Gamma$ be an open conic neighbourhood of $z_0$.
Let $U\in\mathcal{C}^{\infty}(\rr{2d})$ satisfy the following conditions
\begin{align*}
\exists \varepsilon,\ C>0\text{ s.t. }\left| U(x,\xi) \right| &\leq C e^{-\varepsilon (|x|^{1/\theta}+|\xi|^{1/\theta})} \, \qquad (x,\xi) \in \Gamma,\\
\forall \varepsilon,\exists C_\varepsilon>0\text{ s.t. }\left| U(x,\xi) \right| &\leq C_\varepsilon e^{\varepsilon (|x|^{1/\theta}+|\xi|^{1/\theta})} \, \qquad (x,\xi) \in \Gamma^c,
\end{align*}
Let $G\in\GS{\theta}(\rr{2d})$. Then for every open conic neighbourhood $\Gamma^\prime$ of $z_0$ such that $\overline{\Gamma^\prime}\subset\Gamma$ we have
\begin{equation}
\exists \varepsilon,\ C>0\text{ s.t. }\left| (G*U) (x,\xi)\right| \leq C e^{-\varepsilon (|x|^{1/\theta}+|\xi|^{1/\theta})} \, \qquad (x,\xi) \in \Gamma^\prime.
\label{eq:wfstftdef}
\end{equation}
\end{lem}
\begin{proof}
The proof is done by careful splitting of the convolution integral with respect to $\Gamma$. As we already carry out a similar argument in detail in the later proof of Proposition \ref{prop:density}, we omit it here.
\end{proof}
The following Lemma states what happens if the Gaussian window function used in the transform is replaced by another general element in $\GS{1/2}(\RR)$.
\begin{lem}
\label{lemma9.1}
Let $u \in \GSd{\theta}(\RR)$ and let $\Gamma$ be a closed cone in $T^*\RR$ such that $\Gamma \cap \WFt (u)=\emptyset.$
Then for every $g \in \mathcal{S}_{1/2}(\RR)$ there exist positive constants $M,m$ such that 
$$|V_g(u)(x,\xi)| \leq M e^{-m(|x|^{\frac{1}{\theta}}+|\xi|^{\frac{1}{\theta}})}, \quad (x,\xi) \in \Gamma. $$
Conversely, in the previous characterization \eqref{eq:stftdefeq} of the wave front set we can replace $\psi$ by any non-zero window function $g \in \GS{1/2}(\RR)\setminus\{0\}$. 
\end{lem}
\begin{proof}
By \eqref{eq:groch} we write
$$|V_gu |\lesssim | V_\psi g |*|V_\psi u |.$$
Applying Lemma \ref{lem:convest} yields the assertion.
\end{proof}
The following Proposition asserts that for each possible global wave front set, there exists a distribution with such singularities. The construction used is similar to the original one for the classical wave front set in \cite{Hormander0} and has first been used in a similar statement for the corner component of the $\mathcal{S}$-wave front set in \cite{Coriasco1}. For the Gabor wave front set, i.e. in the tempered setting, it has been carried out in \cite{SW}. Here we adapt it to the Gelfand-Shilov context.
\begin{prop}
Let $\Gamma\subset \rr{2d}\setminus\{0\}$ closed, conic. Then there exists $u\in\GSd{\theta}(\RR)$ such that $\WFt(u)=\Gamma$.
\end{prop}
\begin{proof}
Let $(y,\eta) \in \rr{2d} \setminus \{ 0 \}$ and $k \in \no$ be fixed. We define for $x \in \RR$
\begin{equation}\label{eq:fk}
f_k(x;y,\eta) := \exp\left(-\frac{1}{2}|x+k^2 y|^2 + i k^2 \langle x , \eta \rangle \right) \in \mathcal{S}_{1/2}(\RR).
\end{equation}
We may calculate the modulus of its transform:
\begin{equation}
\label{eq:abseq}
\left|T_\phi\left(f_k(\cdot;y,\eta)\right)(x+i\xi)\right|=\frac{c_\phi}{\pi^{d/2}}\exp\left(-\frac{1}{4}|\xi+k^2y|^2-\frac{1}{4}|x+k^2\eta|^2+\frac{|\xi|^2}{2}\right)
\end{equation}
As the $T_\phi\left(f_k(\cdot;y,\eta)\right)$ are analytic functions, we can thus conclude that
$$U(x+i\xi;y,\eta):=\sum_{k=1}^{\infty} T_\phi\left(f_k(\cdot;y,\eta)\right)$$
is an analytic function satisfying \eqref{eq:gsdstf} everywhere, as the sum is bounded on each compactum.\\
Now take a sequence $(y_j,\eta_j)$, dense in $\Gamma\cap S^{d-1}$ and without repetitions\footnote{In case of a finite $\Gamma$ take the finite set of points $\{(y_j,\eta_j)\}=\Gamma\cap S^{d-1}$.} and define $$U(x+i\xi):=\sum_{j=1}^{\infty} 2^{-j} U(x+i\xi;y_j,\eta_j),$$ which again is an analytic function satisfying the estimate in \eqref{eq:gsdstf}.
We thus define $u$ as the unique ultradistribution in $\GSd{\theta}(\RR)$ such that $T_\phi u\equiv U$.\\
If $z_0\notin\Gamma$, we have by a standard scaling inequality for disjoint cones in an open conic neighbourhood $L$ of $\chi^b_\phi(z_0)$: 
$$\forall\ z\in L,\ w\in\Gamma:\quad |\chi^b_\phi(z)-w|\gtrsim |\chi^b_\phi(z)|+|w|$$
Recall that for the standard phase we have $\chi^b_\phi(x+i\xi)=-\xi-ix$. We may thus conclude by \eqref{eq:abseq} that on $L$ we have
\begin{align*}
|U(z)|&\lesssim \sum_{j=1}^\infty\sum_{k=1}^\infty 2^{-j} \exp\left(-\frac{1}{4}|\chi^b_\phi(z)-k^2(y,\eta)|^2+\Phi(z)\right)\\
&\lesssim \sum_{k=1}^\infty \exp\left(-c(k^4+|z|^2)+\Phi(z)\right),
\end{align*}
which proves $z_0\notin\WF(u)$ and thus $\WF(u)\subset\Gamma$.\\
To prove the opposite inclusion, $\Gamma\subset\WF(u)$, consider a fixed $(y_j,\eta_j)$. For that we note that for $m\neq j$ 
\begin{align}
\left|\left(e^{-\Phi}U\right)(k^2\chi^b_\phi(y_j,\eta_j);y_j,\eta_j)\right|&\stackrel{k\rightarrow\infty}{\rightarrow} \frac{c_\phi}{\pi^{d/2}},\\
\left|\left(e^{-\Phi}U\right)(k^2\chi^b_\phi(y_j,\eta_j);y_m,\eta_m)\right|&\stackrel{k\rightarrow\infty}{\rightarrow} 0.
\end{align}
With these identities it is easy to prove that for suitably large $k$ we have $|U(k^2\chi^b_\phi(y_j,\eta_j))|>1/2,$ cf. \cite{SW}, meaning $(y_j,\eta_j)\in\WFt(u)$. As the $(y_j,\eta_j)$ are dense in $\Gamma$, and $\WFt(u)$ is a closed set, this concludes the proof.
\end{proof}

\subsection{Transformation properties}

In this section we list the behaviour of $\WFt$ under unitary transformations associated with linear symplectomorphisms of $T^\ast\RR$, cf. \cite[Proposition 6.7]{Hormander1}.

\begin{prop}
\label{prop:symptrans}
Let $u \in \GSd{\theta}(\RR)$ and let $(x_0,\xi_0) \in \rr{2d}\setminus\{0\}$. Then the following properties hold: \\
i) Let $\rr{d}=\rr{d_v+d_w}$, 
$$(v_0,w_0,\zeta_0,\eta_0) \in \WFt(u) \Leftrightarrow (\zeta_0,w_0,-v_0,\eta_0) \in \WFt (\mathcal{F}_{v\rightarrow \zeta}{u}). $$
ii) Let $A$ be a real, symmetric $d\times d$-matrix. Then 
$$(x_0,\xi_0) \in \WFt (u) \Leftrightarrow (x_0,\xi_0+Ax_0) \in \WFt \left(e^{i x ^t\hskip-2pt Ax}u\right). $$
iii) Given a linear invertible map $A$ on $\RR$ and denoted by $^t \hskip-3pt A$ its transpose, we have
$$(x_0,\xi_0) \in \WFt (u) \Leftrightarrow (A^{-1}x_0, ^t \hskip-3pt A \xi_0) \in \WFt (A^{\ast}u), $$
where $A^{\ast}u(y)=\sqrt{|\textrm{det}A|}u(Ay).$\\
iv)\ $(x_0,\xi_0) \in \WFt (u) \Leftrightarrow (x_0,-\xi_0) \in \WFt (\bar{u}). $ \\
\end{prop}
To each linear symplectomorphism $\chi:T^*\rr{d}\rightarrow T^*\rr{d}$ there exists an associated unitary transform $U:L^2(\rr{2d})\rightarrow L^2(\rr{2d})$, see \cite{Hormander1}. As $i)-iii)$ yield the generators of the symplectic group, we get from Proposition \ref{prop:symptrans} the following Corollary:
\begin{cor}
Let $\chi:T^*\rr{d}\rightarrow T^*\rr{d}$ linear, symplectic, then  
$$(x_0,\xi_0) \in \WFt (u) \Leftrightarrow \chi(x_0,\xi_0) \in \WFt \left(Uu\right), $$
where $U$ is the unitary transform associated to $\chi$.
\end{cor}
\begin{rem}
The preceding corollary underlines the usefulness of the notion of $\WFt(u)$ to describe global propagation of singularities under Schrödinger equations. \\
For ultradistributions of the form $u(t,\cdot)=\mathcal{F}^{-1}\left(e^{i\xi^2t}\mathcal{F}(u_0)\right)\in\GSd{\theta}(\RR)$, $u_0\in\GSd{\theta}(\RR)$, we have
$$(x,\xi)\in\WFt(u_0)\Leftrightarrow (x+t\xi,\xi)\in\WFt(u(t,\cdot)).$$
These distributions solve the homogeneous initial value problem for the Schrödinger equation
\begin{align*}\begin{cases}
-i\partial_t u(t,\cdot)+\Delta u(t,\cdot)&=0\\
\qquad\quad u|_{t=0}&=u_0.
\end{cases}
\end{align*}
The metaplectic invariance of the (Gabor) wave front set was used in \cite{CNR} to study more general Schr\"odinger operators. We note however that the counter-example of \cite[Proposition 4.1]{CNR2} limits the class of interesting operators in the super-exponential setting. For more on propagation of these singularities under Schr\"odinger operators in various functional settings consider \cite{MNS1, MNS2, Mizuhara, Nakamura1}.
\end{rem}
\subsection{Behaviour under operations}
In the following we will study the behaviour of $\WFt(u)$ under operations such as pull-backs, tensor products, etc. For that, we first introduce a notion of continuity on the space of distributions with wave front set in a given cone.

\begin{defn}
Let $\Gamma$ be a closed sub-cone of $T^*(\RR)\setminus 0$. We denote by $\GSdG{\theta}(\RR)$ the space 
$$ \GSdG{\theta}(\RR):=\{u\in\GSd{\theta}(\RR)|\WFt(u)\subset\Gamma\}$$
endowed with the following notion of convergence:\\
We say that 
$u_n\stackrel{\GSdG{\theta}}{\rightarrow} 0$ if
	\begin{enumerate}
		\item $u_n\stackrel{\GSd{\theta}}\rightarrow 0$,
		\item For all $z\in\Gamma^c$ there exists a conic neighbourhood L of $\chi^b_{\varphi}(z)$, $C>0$ and $\varepsilon>0$ such that for all $n\in\mathbb{N}$, $z^\prime\in L$ we have $\left|T_\varphi(u_n)(z^\prime)\right|\leq C_\varepsilon e^{-\varepsilon |z^\prime|^{1/\theta}+\Phi(z^\prime)}$.
	\end{enumerate}
\end{defn}
\begin{prop}\label{prop:density}
$\GS{\theta}(\RR)$ is dense in $\GSdG{\theta}(\RR)$.
\end{prop}
First let us note that an ultradistribution in $\GSd{\theta}(\RR)$ satisfies the estimate in \eqref{eq:gsstf} for the standard phase $\phi$ in a cone $\Gamma\ni\chi^b_\phi(z)$, if and only if the following short time Fourier transform of $u$, which differs from the standard one by only a phase,
\begin{align} \label{eq:stftineq}
Vu(x,\xi)&=\underbrace{2^{-\frac{d}{2}}\pi^{-\frac{3d}{4}}}_{=:c}\langle u,e^{i(x-y)\cdot \xi-\frac{1}{2}|x-y|^2}\  \rangle
\end{align}
satisfies on $\Gamma$ 
\begin{equation}
|Vu(x,\xi)|\leq C e^{-\epsilon (|x|^{1/\theta}+|\xi|^{1/\theta})} 
\end{equation}
For the proof of Proposition \ref{prop:density} we need to understand the transform of a regularizing pseudodifferential operator acting on a distribution $u$ via the action of another pseudodifferential operator on $Vu$. The above transform enjoys the following identity:
\begin{lem}
Let $a\in\GS{\theta}(\rr {2d})$, $u\in\GSd{\theta}(\RR)$. Then we have the following identity
$$V\left(a(y,D_y)u\right)=\wt a(x,\xi,D_x,D_\xi)V\left(u\right),$$
where $\wt a(x,\xi,x^*,\xi^*)=a(x-\xi^*,x^*)$ and $(x^*,\xi^*)$ denotes the covariable to $(x,\xi)$.
\end{lem} 
\begin{proof}
The statement can be verified for $u\in\GS{\theta}(\RR)$ by repeating the computation in \cite[Proposition 3.3.1]{Martinez1}. The assertion for general $u\in\GSd{\theta}(\RR)$ then follows by a density argument.
\end{proof}
\begin{proof}[Proof of Proposition \ref{prop:density}]
Let $u\in\GSdG{\theta}(\RR)$, i.e. $\WFt(u)\subset\Gamma$ and choose $a_\varepsilon(x,\xi)=e^{-\frac{\varepsilon}{2}|x|^2}e^{-\frac{\varepsilon}{2}|\xi|^2}$. Then $a(x,\xi,\varepsilon)\in\GS{\theta}(\rr{2d})$ for every $\varepsilon>0$ and thus also $u_\varepsilon:=a_\varepsilon(x,D_x)u\in\GS{\theta}(\RR)$.\\
It is easy to verify that $a_\varepsilon(x,D_x)u\stackrel{\GSd{\theta}}{\longrightarrow}u$ as $\varepsilon\rightarrow 0.$\\
It remains to show, that for all $(x_0,\xi_0)\in\Gamma^c$ we can find a $\delta,\ C_\delta$ such that in a conic neighbourhood $\Gamma_0$ of $(x_0,\xi_0)$ we have for every $\epsilon>0$
\begin{equation}
Vu_\varepsilon(x,\xi)\leq C_\delta e^{-\delta (|x|^{1/\theta}+|\xi|^{1/\theta})}.
\end{equation}
To do that, we use 
\begin{align*}
\left|Vu_\varepsilon(x,\xi)\right|&=\left|V\left(a_\varepsilon(y,D_y)u\right)(x,\xi)\right|\\
	&=\left|(\wt{a_\varepsilon}(x,\xi,D_x,D_\xi)Vu)(x,\xi)\right|
\\
	&\lesssim \left|\int a_{\varepsilon}(x-\xi^*,x^*) e^{i(x-x^\prime)x^*+i(\xi-\xi^\prime)\xi^*} Vu(x^\prime,\xi^\prime)\ d x^\prime d x^* d \xi^\prime d \xi^* \right|
\\
	&\lesssim \varepsilon^{-d/2}\int a_{\varepsilon^{-1}}(x-x^\prime,\xi-\xi^\prime)   \left|Vu(x^\prime,\xi^\prime)\right| \ d x^\prime d \xi^\prime 
\\
	&=\varepsilon^{-d/2} \left(a_{\varepsilon^{-1}}*|Vu|\right)
\end{align*}
In order to estimate the convolution, we split the integral in two parts:\\
As $\WFt(u)\subset\Gamma$, we pick an open cone $\Gamma_1\ni (x_0,\xi_0)$, such that $\overline{\Gamma_1}\cap\Gamma=\emptyset$. We can then pick an intermediate closed cone $\Gamma_2$ such that $\Gamma_1\subset\Gamma_2$ and $\Gamma_2\cap\Gamma=\emptyset$. We then write, with $z=(x,\xi)$:
\begin{align*}
C \varepsilon^{-d/2} \left(a_{\varepsilon^{-1}}*|Vu|\right)(z)&=C\varepsilon^{-d/2} \int_{\Gamma_2} e^{\frac{-1}{2\varepsilon}|z-w|^2}|Vu(w)|\ d w\\
&\quad+C\varepsilon^{-d/2} \int_{\Gamma_2^c} e^{\frac{-1}{2\varepsilon}|z-w|^2}|Vu(w)|\ d w\\
&=I_{\Gamma_2}(z)+I_{\Gamma_2^c}(z).
\end{align*}
Let us first study $I_{\Gamma_2^c}$. There we have, for $z\in \Gamma_1$, due to a standard scaling estimate for disjoint cones $|z-w|\gtrsim |z|+|w|$ and therefore for some constants $\delta_i>0$
$$e^{\frac{-1}{2\varepsilon}|z-w|^2}\leq e^{\frac{-\delta_1}{2\varepsilon}|z|^2}e^{\frac{-\delta_1}{2\varepsilon}|w|^2},$$ 
and consequently
$$|I_{\Gamma_2^c}(z)|\lesssim e^{-\delta_2|z|^2}.$$
On $\Gamma_2$ we can assume, due to compactness of $\Gamma_2\cap S^{d-1}$, that there exist a single constant $\delta_3>0$ such that
$$Vu(w)\leq e^{-\delta_3 |w|^{1/\theta}}.$$
Using this, we conclude that
$$|I_{\Gamma_2}(z)|\lesssim e^{-\delta_4|z|^{1/\theta}}$$
and thus
\begin{align*}
\left|Vu_\varepsilon(x,\xi)\right| \lesssim \varepsilon^{-d/2} \left(a_{\varepsilon^{-1}}*|Vu|\right)(z)&\lesssim e^{-\delta_2|z|^2}+e^{-\delta_4|z|^{1/\theta}} \lesssim e^{-\delta|z|^{1/\theta}},
\end{align*}
which proves the assertion.
\end{proof}

We now study tensor products and pullbacks of ultradistributions and their resulting wave front set, following \cite{Hormander1}. 

\begin{prop}[Behaviour under tensor products]
\label{prop:tens}
Let $u\in\GSd{\theta}_{,\Gamma_1}(\rr{d_1})$, $v\in\GSd{\theta}_{,\Gamma_2}(\rr{d_2})$. Define 
$$\Gamma= \left((\Gamma_1\cup\{0\})\times(\Gamma_2\cup\{0\})\right)\setminus\{(0,0)\} \subset \rr{d_1+d_2}.$$
Then $u\otimes v\in\GSdG{\theta}(\rr{d_1+d_2}).$
\end{prop}

\begin{proof}
This is a consequence of $T_{\phi}(u\otimes v)=T_{\phi}(u)\otimes T_{\phi}(v)$.
\end{proof}

\begin{thm}[Behaviour under the pullback by a linear map]
\label{thm:pullb}
Let $A$ be a linear map $\rr{m}\rightarrow\RR$. Let $\Gamma$ be a closed cone such that
\begin{equation}
\label{eq:wfcrit}
\Gamma\cap\{(0,\xi)| ^t\hskip-2pt A\xi=0\}=\emptyset.
\end{equation}
Then the pull-back $A^*:\GS{\theta}(\RR)\rightarrow\GS{\theta}(\rr{m})$  can be uniquely extended to a continuous map $\GSdG{\theta}(\RR)\rightarrow\GSd{\theta}_{,A^*\Gamma}(\rr{m}),$
where 
$$A^*\Gamma=\{(x, ^t \hskip-3pt A\xi)|(Ax,\xi)\in\Gamma\}.$$
\end{thm}
\begin{proof}
For the purpose of self-containedness, we give a shortened proof with respect to \cite[Proposition 6.15]{Hormander1}.\\ 
Due to Proposition \ref{prop:symptrans} it suffices to show this for the maps $\iota:\rr{d-1}\rightarrow\RR$ $y\mapsto(y,0)$ and $\pi:\rr{d+1}\rightarrow\RR$ $(x,x^\prime)\mapsto x$. For the second case, we can define $\pi^*u=u\otimes \mathbbm{1}$.\\
We are therefore reduced to the case of $\iota$. Formally, we want to define $\langle u,\ f\otimes\delta(x_d)\rangle$. In view of Lemma \ref{lem:moyal}, we therefore make the following formal calculation, with the notation $\phi=\phi_d=\phi_{d-1}+\phi_{1}$ and $(z_1,\dots,z_{d-1},z_d)=(z^\prime,z_d)$,
\begin{multline}
\langle \pi^*u,\ f \rangle = \int_{\CC} T_{\phi_d} u(z)  T_{-\bar{\phi}_d} (f\otimes \delta(x_d))(z) e^{-2\Phi(z)}d\lambda(z)\\ 
=\int_{\CC} T_{\phi_d} u(z)  T_{-\bar{\phi}_{d-1}} f(z^\prime) T_{-\bar{\phi}_{1}}(\delta(x_d))(z_d) e^{-2\Phi_{d-1}(z^\prime)-2\Phi_{1}(z_d)}d\lambda(z)\\
=c_{\phi_1}\int_{\cc{d-1}\times\cc{1}} T_{\phi_d} u(z^\prime,z_d)  T_{-\bar{\phi}_{d-1}} f(z^\prime) e^{-2\Phi_{d-1}(z^\prime)-2\Phi_{1}(z_d)}d\lambda(z^\prime ) d\lambda (z_d).
\label{eq:pullbackdef}
\end{multline}
Again in light of Lemma \ref{lem:moyal} we consider the expression
\begin{equation}
\label{eq:Udef}
U(z^\prime)=c_{\phi_1}\int_{\cc{} } T_{\phi_d} u(z^\prime,z_d) e^{-2\Phi_{1}(z_d)} d\lambda (z_d).
\end{equation}
In the situation $A=\iota$ the condition \eqref{eq:wfcrit} takes the form
$$\WFt(u)\cap\{(0,\xi)|\xi_1=\dots=\xi_{d-1}=0\}=\emptyset \Rightarrow (0,\dots,0,\pm 1) \notin \WFt (u).$$
We calculate $\chi^b_\phi(0+i(0,\dots,0,\pm 1))=(0,\dots,0,\mp 1)$ and conclude that there exist open cones $\Gamma_{\pm}\ni (0,\dots,0,\pm 1)$ such that $T_{\phi_d}u$ satisfies the estimate in \eqref{eq:gsstf} on $\Gamma_{+} \cup \Gamma_-$. 
Using \eqref{eq:gsdstf} and $\Phi_{1}(z_d)=\frac{1}{2}|\Im(z_d)|^2$ we conclude that the integrand of \eqref{eq:Udef} is bounded by
$$e^{-\varepsilon |\Re(z_d)|^{1/\theta}-\frac{1}{2}|\Im(z_d)|^2}e^{\delta |z^\prime|^{1/\theta}+\Phi_{d-1}(z^\prime)}$$
for some $\varepsilon>0$ and any $\delta>0$.
Therefore the integral \eqref{eq:Udef} converges for any $z^\prime$, and yields an entire function satisfying
$$\forall \varepsilon>0\ \exists C_\varepsilon>0\text{ s.t.}\left| U(z^\prime)\right|\leq  C_\varepsilon e^{\varepsilon |z^\prime|^{1/\theta}+\Phi(z^\prime)}.$$
We can therefore define $A^*u$ as the ultradistribution $v \in \GSd{\theta}(\rr{d-1})$, granted by Theorem \ref{thm:wfprop}, such that $T_{\phi_{d-1}}v=U.$\\
The estimate for the resulting wave front set follows by careful splitting of the integral into regions where the integrand satisfies the stronger estimates \eqref{eq:gsstf}. The continuity is immediate from dominated convergence of the integrals.
\end{proof}
With this notion of pullback and the tensor product it is now possible to introduce products, convolutions, restrictions and pairings of tempered ultradistributions.
\begin{cor}[Products and convolutions]
Let $\Gamma_1,\ \Gamma_2\subset T^*\RR\setminus\{0\}$, 
$$\Gamma_3=\{(x,\xi+\eta)|(x,\xi)\in\Gamma_1\textrm{ and }(x,\eta)\in\Gamma_2\}\cup\Gamma_1\cup\Gamma_2.$$
Then the product of two ultradistributions $u\in\GSd{\theta}_{,\Gamma_1}{\left(\RR\right)}$, $v\in\GSd{\theta}_{,\Gamma_2}{\left(\RR\right)}$ is well-defined under the assumption that $0\notin\Gamma_3$, i.e. 
$$(0,\xi)\in\Gamma_1\Rightarrow (0,-\xi)\notin\Gamma_2.$$
Under these assumptions, we have the inclusion $\WFt(u\cdot v)\subset\Gamma_3$ and the product is a continuous mapping $\GSd{\theta}_{,\Gamma_1}(\RR)\times \GSd{\theta}_{,\Gamma_2}(\RR)\rightarrow\GSd{\theta}_{,A^*\Gamma}(\RR),$\\
Similarly, the convolution of two ultradistributions $u,v\in\GSd{\theta}{\left(\RR\right)}$ is well-defined under the assumption that $$(\xi,0)\in\WFt(u)\Rightarrow (-\xi,0)\notin\WFt(v).$$
We then have the inclusion $$\WFt(u * v)\subset\{(x+y,\xi)|(x,\xi)\in\WFt(u)\cap\{0\}: (y,\xi)\in\WFt(u)\cap\{0\}\}.$$
\end{cor}
\begin{proof}
Use \thmref{thm:pullb} and Proposition \ref{prop:tens} to define the product of two distributions $u$ and $v$ by $(u\cdot v)(x)=\delta^* \big(u(x)\otimes v(y)\big)$ where $\delta$ is the diagonal map $x\mapsto(x,x)$.\\
The statement about convolution follows directly by Fourier transformation and $i)$ of Proposition \ref{prop:symptrans}.
\end{proof}
\begin{cor}[Pairings of ultradistributions]
\label{cor:pair}
Under the assumption that $(x,\xi)\in\WFt(u)\Rightarrow(x,-\xi)\notin\WF(v)$ we can define the pairing of $u$ and $v$ as the unique continuous extension of the pairing of two test functions.
\end{cor}
\begin{proof}
We define the pairing as the image of $\mathcal{F}(u\cdot v)$ under the pull-back via $0\hookrightarrow \RR$.
\end{proof}

\section{Microlocality and microellipticity properties}\label{micro}

In this section we prove microlocality and microellipticity properties for several classes of operators with respect to $\WFt(u)$. 

\subsection{General operators}
Using the techniques of Corollary \ref{cor:pair}, pairing with respect to only a subset of the variables, we can first estimate the wave front set of an operator $K$ applied to an ultradistribution in terms of the wave front set of its kernel $\mathcal{K}$, see \cite[Proposition 2.11]{Hormander1}.
\begin{prop}[Microlocal mapping properties in terms of the kernel]
Let $\mathcal{K}\in\GSd{\theta}(\rr{d_2+d_1})$ and $K$ the associated operator $K:\GS{\theta}(\rr{d_1})\rightarrow\GSd{\theta}(\rr{d_2})$. Then $K$ can be extended to all $\GSdG{\theta}(\rr{d_1})$ such that
$$\Gamma\cap\{(y,\eta)|(0,y,0,-\eta)\in\WFt(\mathcal{K})\}= \emptyset.$$
For $u$ in $\GSdG{\theta}(\rr{d_1})$  we then have the estimate
$$\WFt(Ku)\subset\{(x,\xi)|(x,0,\xi,0)\in\WFt(\mathcal{K})\}\cup(\WFt)^\prime(K)\circ\WFt(u)$$
where $(\WFt)^\prime(K)$ is the relation given by $\big\{\big((x, \xi); (y,\eta)\big)|(x,y,\xi,-\eta)\in\WFt(\mathcal{K})\big\}$.
In particular an operator $K:\GSd{\theta}(\RR)\rightarrow\GSd{\theta}(\RR)$ is microlocal if $\WFt(\mathcal{K})$ only contains elements of the form $(x,x,\xi,\xi)$.
\end{prop}

Let us now consider special operators. 

\subsection{Localization operators}

Localization operators, or Anti-Wick quantized operators, have appeared in many contexts, ranging from Quantum field theory to signal analysis. Quite recently, localization operators in the setting of various function spaces have been an active field of research. For a history and survey on the subject, consider \cite{CGR} and the references therein. The function spaces considered include Bargmann-Fock spaces, modulation spaces with exponential weights \cite{CPRT2, Toft2} and Gelfand-Shilov spaces \cite{Toft}, in particular quasi-analytic ones \cite{CPRT}, using in particular the good transformation behaviour of these spaces with respect to the short time Fourier transform. It is therefore only natural to consider their microlocal properties with respect to our global wave front set. We will do so in this section, proving a microlocality result.

\begin{defn}
Let $a\in\GSd{\theta}(\rr{2d})$. The localization operator $A^\psi_a$ with respect to the standard window $\psi$ with symbol $a$ is weakly defined by (for $u,\ v\in\GS\theta(\RR)$):
$$\langle A^\psi_a u,v\rangle = \langle a,V_\psi u \cdot V_\psi v\rangle.$$
\end{defn}

We have already noted that $(x_0,\xi_0) \notin \WFt(u)$ if the short-time Fourier transform $V_\psi u(x,\xi)$ satisfies \eqref{eq:wfstftdef} on an open cone $\Gamma\ni(x_0,\xi_0)$. We now recall the result of \cite[Proposition 5.11]{Toft}, stating that if $\theta\neq 1/2$ and   
$a\in L^\infty_\mathrm{loc}(\rr{2d})$ satisfies that $\forall\varepsilon>0$ there exists $C_\varepsilon>0$ such that 
\begin{equation}
\label{eq:GSmultiplier}
|a(x,\xi)|\leq C_\varepsilon e^{\varepsilon(|x|^{1/\theta}+|\xi|^{1/\theta})},
\end{equation}
then $A^\psi_a$ is continuous on both $\GS{\theta}(\RR)$ and $\GSd{\theta}(\RR)$. We can give a microlocal improvement of this statement, by estimating $V_\psi(A^\psi_a u)$ in terms of $V_\psi u$ and $a$:
\begin{align*}
V_\psi A^\psi_a u(v,\eta)&=V_\psi V_\psi^*(a\cdot(V_\psi u))(v,\eta)\\
	&=\int e^{-\frac{1}{2}|v-y|^2+i \langle y,\eta \rangle}e^{-\frac{1}{2}|x-y|^2-i \langle \xi, y \rangle}a(x,\xi)V_\psi u(x,\xi)\ dx d\xi dy\\
	&=\int e^{-|y-\frac{x+v}{2}|^2}e^{-\frac{1}{4}|x-v|^2}e^{i\langle y,\eta-\xi \rangle}a(x,\xi)V_\psi u(x,\xi)\ dx d\xi dy\\
	&= 2^{-d/2}\int e^{-\frac{1}{4}|\eta-\xi|^2}e^{-\frac{1}{4}|x-v|^2}e^{\frac{i}{2} \langle x+v,\eta-\xi \rangle}a(x,\xi)V_\psi u(x,\xi)\ dx d\xi 
\end{align*}
We can conclude that 
$$|V_\psi A^\psi_a u(x,\xi)|\lesssim \left(e^{-\frac{|\cdot|^2}{4}}*\left|(a\cdot V_\psi u)\right|\right)(x,\xi).$$
Using \lemref{lem:convest}, we obtain the following:
\begin{prop}[Microlocality of localization operators]
Let $\theta>1/2$ and $A^\psi_a $ be a localization operator with symbol $a\in L^\infty_\mathrm{loc}(\rr{2d})$ satisfying \eqref{eq:GSmultiplier}. Then we have $\WFt(A^\psi_a u)\subset \WFt(u)$.
\end{prop}
We note that the method of the proof is not limited to the quasianalytic case but it can applied also for $\theta \geq 1.$
\begin{rem}
We can conclude that if $A^\psi_a$ is an invertible operator then we have the equality
$$\WFt(A^\psi_a u)= \WFt(u).$$
\end{rem}
\subsection{Ultradifferential operators}
In general, localization operators can have symbols that are not analytic. Motivated by \thmref{thm:wfprop} we can instead consider the following class of operators, where the coefficients are multipliers and partial derivatives appear.
\\
Denote the space of analytic functions $h:\CC\rightarrow \co$ satisfying
\begin{equation}
\label{eq:hsdef}
\forall \varepsilon>0\ \exists D_\varepsilon>0\text{ s.t.} |h(z)|\leq  D_\varepsilon\ e^{\varepsilon |z|^{1/\theta}}
\end{equation} 
by $MS^\theta(\CC)$. 
\begin{thm}
Let $\{h_\alpha\}$ be a family of elements of $MS^\theta(\CC)$ such that the constants $D_{\epsilon,\alpha}$ in \eqref{eq:hsdef} satisfy 
\begin{equation}
\label{eq:coeffineq}
D_{\epsilon,\alpha} \leq \frac{D^\prime_{\varepsilon}\varepsilon^{|\alpha|}}{\alpha!^\theta}
\end{equation}
for any $\varepsilon$ and for some constant $D^\prime_\varepsilon>0.$ Then we can define the following operator
$$Au(z)=\sum_{\alpha\in\no^{d}}  T_\varphi^* h_\alpha(z) (\partial_z^\alpha T_\varphi u)(z)$$ 
as a continuous map $\GSdG\theta (\RR)\rightarrow\GSdG\theta (\RR)$. \\
In particular it fulfills $\WFt(Au)\subset\WFt(u)$ and maps $\GS{\theta}$ into itself.
\end{thm}
We can, by hand, compute the following relations for the transform $T_\varphi u$ of $u\in\GS{\theta}(\RR)$ with the standard phase:
\begin{align*}
zT_\varphi u(z)&=T_\varphi\big((-iy+i\partial_y)u\big)(z)\\
\partial_zT_\varphi u(z)&=T_\varphi\big((i\partial_y)u\big)(z)
\end{align*}
Therefore ultradifferential operators with polynomial coefficients can be understood as a subclass of the operators just considered, if their coefficients satisfy \eqref{eq:coeffineq}. In particular, differential operators with polynomial coefficients are microlocal. For these, we can also get the reverse wave front set inclusion in terms of the principal symbol of the operator, known as microellipticity, following \cite{Hormander1}.
\subsubsection{Microellipticity of differential operators}
Let $P=p(x,D)$ be a differential operator with polynomial coefficients. We can write it as follows
$$p(x,D) = \sum_{|\alpha|+|\beta| \leq m} c_{\alpha \beta} x^{\beta}D^{\alpha}.$$ In the sequel we shall denote by $p_m(x,\xi)$ the following principal symbol
$$p_m(x,\xi)= \sum_{|\alpha|+|\beta| = m} c_{\alpha \beta} x^{\beta}\xi^{\alpha} $$ which is homogeneous of order $m$ in $(x,\xi)$ and define the characteristic set of $P$ as follows:
$$\mathrm{Char}(P)= \{(x,\xi) \in T^{\ast}\RR: p_m(x,\xi) =0 \}.$$ We have the following result.

\begin{thm}\label{microellipticity}
Let $u \in \GSd{\theta}(\RR)$ and let $p(x,D) u =f.$ Then the following inclusions hold:
\begin{equation}
\label{microell}
 \WFt(u) \subset \WFt(Pu) \cup \mathrm{Char}(P).
\end{equation}
\end{thm}
\begin{rem}
We observe that Theorem \ref{microellipticity} represents a generalization of Theorem 1.1 in \cite{CGR1} for the case $\mu=\nu=\theta.$ 
As a matter of fact, if the operator $P$ is globally elliptic, i.e. $\mathrm{Char} (P) =\emptyset$, we obtain that $\WFt(u) = \WFt(Pu)$.
In particular, Corollary \ref{cor:wfempty} implies that if $Pu \in \GS{\theta}(\RR),$ then $u \in \GS{\theta}(\RR).$
\end{rem}
\noindent
\textit{Proof of Theorem \ref{microellipticity}.}
We follow the outline of a proof given in \cite{Hormander1} for the case $\theta=1/2$, which uses the idea from \cite{Hormander0} of estimating $| \langle u, e^{-i\langle y, \xi \rangle -\frac{1}{2}|x-y|^2} \rangle |$ in terms of $| \langle Pu, w \rangle |$ by constructing an approximate solution for the equation $P^*w=e^{-i\langle y, \xi \rangle -\frac{1}{2}|x-y|^2}$.\\
Let $(x_0, \xi_0) \notin \WFt(Pu)\cup \mathrm{Char}(P).$ By this assumption there exists a closed conic set $\Gamma \subset T^{\ast}(\RR)$ such that $p_m(x,\xi) \neq 0$ on $\Gamma$ and there exist positive constants $C, \delta$ such that
\begin{equation}
\label{fcondition}
| \langle f, e^{-i\langle y, \xi \rangle -\frac{1}{2}|x-y|^2} \rangle | \leq C e^{-\delta(|x|^{\frac{1}{\theta}}+|\xi|^{\frac{1}{\theta}})}
\end{equation}
for any $(x,\xi) \in \Gamma.$ We consider, for fixed $(x,\xi) \in \Gamma$ the equation 
\begin{equation}
\label{eq:weq}
p^{\ast}(y,D_y) w_{x,\xi}(y)= e^{-i\langle y, \xi \rangle -\frac{1}{2}|x-y|^2},
\end{equation}
where $p^{\ast}(y,D_y)$ denotes the adjoint of $P$ and look for suitable approximate solutions. 
For that we turn it into a more standard form which can then be approximately solved by a Neumann series. 
 Setting $z=y-x$, we can re-write the equation \eqref{eq:weq} in the following form 
\begin{equation} \label{2}
p^{\ast}(x+z, D_z-\xi+iz)W_{x,\xi}(z)=1,
\end{equation}
with $$W_{x,\xi}(z)= e^{i\langle x+z, \xi \rangle+\frac{1}{2}|z|^2}w_{x,\xi}(x+z).$$
Since $p_m(y,-\eta)$ is the principal symbol of $P^{\ast},$ then $p^{\ast}(x+z, -\xi+iz)-p_m(x+z,\xi-iz)$ is a polynomial of degree strictly less than $m$ in $(x,\xi,z).$ Moreover, denoting $\rho= \sqrt{|x|^2+|\xi|^2},$ we have that, since $p_m(x,\xi) \neq 0 $ on $\Gamma$,  then there exist positive constants $M,c>0$ such that
$$|p^{\ast}(x+z,-\xi+iz)| \geq c \rho^m$$
on the set $$\Gamma_{c,M}=\{ (x,\xi,z) \in \mathbb{R}^{3d}: (x,\xi) \in \Gamma , \rho>M , |z| <c \rho\}.$$ 
In other words, $(1/p^{\ast})(x+z,-\xi+iz)$ behaves like a symbol of order $-m$ for $(x,\xi, z) \in \Gamma_{c,M}.$ Setting now 
\begin{equation}
\label{G}
	G_{x,\xi}(z)= p^{\ast}(x+z,-\xi+iz)W_{x,\xi}(z),
\end{equation}
the equation \eqref{2} takes the form
\begin{equation} 
\label{3}
	G_{x,\xi}(z)-RG_{x,\xi}(z)=1 
\end{equation}
for some operator $R= \sum\limits_{|\alpha| \geq 0}R_{\alpha}(x,\xi,z)D_z^{\alpha},$ where $R_{\alpha}$ is an operator of order $-|\alpha|$ with analytic coefficients on $\Gamma_{c,M}$. Moreover, for $z=\rho \zeta,$ we have that 
$$R= \sum_{|\alpha| \neq 0}R_{\alpha}(x,\xi,\rho \zeta)\rho^{-|\alpha|}D_\zeta^{\alpha}$$ with $|R_{\alpha}(x,\xi,\rho \zeta)\rho^{-|\alpha|}| \leq C \rho^{-2|\alpha|}.$ \\ 
Let us now consider the equation \eqref{3}. To go further with the proof we need to study the equation \eqref{3} for $z \in \CC,$ that is for $y \in \CC.$ A formal solution for the equation would be given by the Neumann series $\sum\limits_{j \geq 0}R^j 1.$ For $N \in \mathbb{N},$ let $G_{x,\xi}^N(z)$ be the sum of all terms in the series involving derivatives with respect to $z$ of order less or equal than $N$. Then we have
$$G_{x,\xi}^N(z)-RG_{x,\xi}^N(z)=1-e_{x,\xi}^N(z)$$ 
for some functions $e_{x,\xi}^N(z).$ Let now, as an approximate solution candidate for \eqref{eq:weq},
$$w_{x,\xi}^N(y)= e^{-i\langle y, \xi \rangle -\frac{1}{2}|x-y|^2}\frac{G_{x,\xi}^N(y-x)}{p^{\ast}(y,-\xi+i(y-x))}.$$ 
Then we obtain 
\begin{equation}
\label{4} 
p^{\ast}(y,D_y)w_{x,\xi}^N(y)= e^{-i\langle y, \xi \rangle -\frac{1}{2}|x-y|^2} \big(1-e_{x,\xi}^N(y-x)\big).
\end{equation}
 Arguing as in \cite{Hormander0}, we obtain the following estimates:
\begin{equation} \label{estGN}
|D_z^{\beta}G_N^{x,\xi}(z)| \leq C^N N^{|\beta|} , \qquad |\beta| \leq N < \rho^2
\end{equation}
\begin{equation} \label{esteN}
|D_z^{\beta}e_N^{x,\xi}(z)| \leq C^N N^{N+|\beta|} \rho^{-2N}, \qquad |\beta| \leq N.
\end{equation}
For every $(x, \xi) \in \Gamma, |z| <c\rho$ and with $c\rho >1.$
Let now $\chi \in C^{\infty}_0(\CC)$ such that $\chi(z)=1$ for $|z|<c/2$ and $\chi(z)=0$ for $|z|>c$ and consider the function
$$h^N_{x,\xi}(z)= \chi\left(\frac{z}{\rho}\right)e^{-\frac{1}{2}\langle z,z \rangle} \frac{G^N_{x,\xi}(z)}{p^*(x+z,-\xi+iz)}.$$
We observe that 
\begin{equation}
\label{eq:hN}
|h^N_{x,\xi}(z)| \leq C^{N+1} \rho^{-m}\exp \left( -\frac{1}{2}|\Re z|^2+ \frac{1}{2}|\Im z|^2 \right), \quad z \in \CC.
\end{equation}
Due to the cut-off function involved, $h^N_{x,\xi}$ is not holomorphic. We now construct a function $\wt h^N_{x,\xi}$ which shares Gaussian decay
such that $h^N_{x,\xi}- \wt h^N_{x,\xi}$ is holomorphic.\\
By \eqref{estGN} for $|\beta|=1$ the coefficient functions of $\bar\partial_z h^N_{x,\xi}$ satisfy, for $c\rho/2 <|z|<c\rho$:
\begin{multline}\label{estbar}
\left|\frac{\partial}{\partial \bar{z}_j}h^N_{x,\xi}(z)\right| \leq C^{N+1}\rho^{-m}\exp \left( -\frac{1}{2}|\Re z|^2+ \frac{1}{2}|\Im z|^2 \right) \\ \leq C^{N+1}\rho^{-m}\exp \left( -\frac{1}{4}|\Re z|^2+ \frac{3}{4}|\Im z|^2-\frac{1}{4}|z|^2 \right) \\ \leq 
C^{N+1}\rho^{-m}\exp \left( -\frac{1}{4}|\Re z|^2+ \frac{3}{4}|\Im z|^2 - \frac{c^2\rho^2}{16} \right).
\end{multline}
From the last estimate we obtain that
\begin{equation}\label{L2phi}
\sum_{j=1}^d\int_{\CC}\left|\frac{\partial}{\partial \bar{z}_j}h^N_{x,\xi}(z)\right|^2 e^{-\kappa(z)}\, d\lambda(z) \leq2 C'' C^{2N}\rho^{-2m+2d} e^{-\frac{c^2\rho^2}{8}}
\end{equation}
with $\kappa(z)= -\frac{1}{2}|\Re z|^2 + \frac{3}{4}|\Im z|^2.$
Therefore the coefficient functions of $\bar{\partial}h^N_{x,\xi}$ are elements of $L^2(\CC, e^{-\kappa}d\lambda).$
We observe that $\kappa(z)$ is plurisubharmonic, see \cite[Definition 2.6.1]{Hormander-1}. As $\bar{\partial}^2=0$, we can apply Theorem 4.4.2 in \cite{Hormander-1} to the equation $\bar{\partial}v=\bar{\partial}h^N_{x,\xi}$. It therefore admits a solution $\tilde{h}^N_{x,\xi}(z)$ such that
$$\int_{\CC} |\tilde{h}^N_{x,\xi}(z)|^2 e^{-\kappa(z)}(1+|z|^2)^{-2}\, d\lambda(z) \leq \frac{C''}{2} C^{2N} \rho^{2d-2m}e^{-\frac{c^2\rho^2}{8}}.$$
By the Cauchy's inequalities we then get
\begin{equation}\label{estbartilde}
|\tilde{h}^N_{x,\xi}(z)|\leq C''' C^N \rho^{d-m}e^{-\frac{c^2\rho^2}{16}}(1+|z|)^d \exp \big( -\frac{1}{4}|\Re z|^2+ \frac{3}{4}|\Im z|^2 \big), \, z \in \CC.
\end{equation}
By construction, the function $H^N_{x,\xi}(z)= h^N_{x,\xi}(z)- \tilde{h}^N_{x,\xi}(z)$ is holomorphic on $\CC$ because $\bar{\partial}H^N_{x,\xi}(z)=0$. Moreover, the estimates \eqref{eq:hN} and \eqref{estbartilde} imply that its restriction on the real domain is in $\mathcal{S}_{1/2}(\RR)$. Then taking now $z=y-x \in \mathbb{R}^d$,  we can write
\begin{multline} \label{decomposition}
\langle u, e^{-i \langle \cdot , \xi \rangle - \frac{1}{2}|\cdot -x |^2 }\rangle = \langle f, e^{-i \langle \cdot , \xi \rangle } H^N_{x,\xi}(\cdot -x) \rangle \\+  \langle u, e^{-i \langle \cdot, \xi \rangle - \frac{1}{2}|\cdot -x |^2 }\rangle - \langle f, e^{-i \langle \cdot, \xi \rangle } H^N_{x,\xi}(\cdot -x) \rangle  \\ = \langle f, e^{-i \langle \cdot , \xi \rangle } H^N_{x,\xi}(\cdot -x) \rangle + \langle u , e^{-i \langle \cdot , \xi \rangle - \frac{1}{2}|\cdot -x |^2 } - p^*(y,D_y)(e^{-i \langle \cdot , \xi \rangle } H^N_{x,\xi}(\cdot -x)) \rangle.
\end{multline}
To conclude the proof we need to estimate properly the two terms in the right-hand side of \eqref{decomposition}.
Concerning the first one we observe that since $H^N_{x,\xi} \in S_{1/2}(\RR),$ then its short time Fourier transform is
in $S_{1/2}(\mathbb{R}^{2d}).$ In particular we have
$$|V_{\psi}H_{x,\xi}^N(x,\xi)| \leq C_1 C^N e^{-\delta (|x|^2+|\xi|^2)}, \qquad (x,\xi) \in \mathbb{R}^{2d},$$ 
for some constants $C_1 >0, 0<\delta <1$ independent of $N$ and where $C$ is the same constant appearing in the estimates \eqref{esteN}, \eqref{estGN}, \eqref{eq:hN}, \eqref{estbartilde}.
Choosing now $N$ such that 
\begin{equation}
\label{choiceN} \frac{\delta \rho^2}{Ce}-1 \leq N \leq \frac{\delta \rho^2}{Ce},
\end{equation}
with $\rho=\sqrt{|x|^2+|\xi|^2}$, we obtain 
$$|V_{\psi}H_{x,\xi}^N(x,\xi)| \leq C_1^\prime  e^{\frac{\delta\rho^2}{e}-\delta \rho^2}= C_1^\prime e^{-\delta(1-e^{-1}) (|x|^2+|\xi|^2)}.$$
Then by the condition \eqref{fcondition} and by Lemma \ref{lemma9.1} we get
\begin{multline*}| \langle f, e^{-i \langle \cdot , \xi \rangle } H^N_{x,\xi}(\cdot -x) \rangle| \leq |V_{\psi}H_{x,\xi}^N(x,\xi)| \ast |\langle f, e^{-i \langle \cdot , \xi \rangle -\frac{1}{2}|\cdot -x|^2}  \rangle| \\ \leq C_2 e^{-\delta_2 (|x|^{1/\theta}+|\xi|^{1/\theta})} \end{multline*}
for some positive constants $C_2, \delta_2$ and for $(x,\xi) \in \Gamma$.
Now we want to prove that the second term in the right-hand side of \eqref{decomposition}  is $\mathcal{O}(e^{-\delta_3 \rho^2})$ for some $\delta_3 >0$ uniformly with respect to $N$. To do this we need to estimate the function 
\begin{equation}
\label{longterm}e^{-i \langle y , \xi \rangle - \frac{1}{2}|y -x |^2 } - p^*(y,D_y)(e^{-i \langle y , \xi \rangle } H^N_{x,\xi}(y -x))
\end{equation}
and its derivatives.
For $|y-x| <c\rho/2$ we have $e^{-i\langle y, \xi\rangle }h_{x,\xi}^N (y-x)=w_{x,\xi}^N(y-x)$. Then 
\begin{multline*}
e^{-i \langle y , \xi \rangle - \frac{1}{2}|y -x |^2 } - p^*(y,D_y)(e^{-i \langle y , \xi \rangle } H^N_{x,\xi}(y -x))
\\ =  e^{-i \langle y , \xi \rangle - \frac{1}{2}|y -x |^2 }e^N_{x,\xi}(y-x)+p^{\ast}(y,D_y)e^{-i\langle y,\xi\rangle}
\tilde{h}_{x,\xi}^N(y-x).
\end{multline*}
Now, by \eqref{esteN} and \eqref{choiceN} we have
$$|e^{-i \langle y , \xi \rangle - \frac{1}{2}|y -x |^2 }e^N_{x,\xi}(y-x)| \leq |e^N_{x,\xi}(y-x)| \leq C^N N^N \rho^{-2N}
\leq e^{-N} \leq e^{1-\frac{\delta\rho^2}{Ce}},$$
since $0<\delta <1$, whereas the term $p^{\ast}(y,D_y)e^{-i\langle y,\xi\rangle}
\tilde{h}_{x,\xi}^N(y-x)$ satisfies a similar bound by \eqref{estbartilde}.
For $|y-x| \geq c\rho/2$, we obviously have $e^{-i \langle y , \xi \rangle - \frac{1}{2}|y -x |^2 }= \mathcal{O}(e^{-c^2\rho^2/8}).$
Moreover, arguing as for \eqref{estbar}, we obtain for some $\gamma>0$ that 
$$|H_{x,\xi}^N(y-x)| \leq C_3 e^{-\delta_3 \rho^2}\exp \big( -\gamma |y -x|^2 \big), $$
and a similar estimate holds for $p^*(y,D_y)(e^{-i \langle y , \xi \rangle } H^N_{x,\xi}(y -x))$. In conclusion, for $|y-x| \geq c\rho/2$ we obtain
\begin{multline*}|e^{-i \langle y , \xi \rangle - \frac{1}{2}|y -x |^2 } - p^*(y,D_y)(e^{-i \langle y , \xi \rangle } H^N_{x,\xi}(y -x)|  \\ \leq C_4
e^{-\delta_3 \rho^2}\exp \big( -\gamma |y -x|^2 \big).
\end{multline*}
The estimate of the derivatives of the function \eqref{longterm} can be obtained by estimating the derivatives of its entire extension by 
Cauchy's inequalities arguing as in the proof of Lemma \ref{lem:cauchyineq}. The details are left to the reader, cf. \cite{Hormander1}.

\qed

\section*{acknowledgements}

We are grateful to Profs. D. Bahns, E. Cordero, L. Rodino, J. Toft, P. Wahlberg and I. Witt for valuable advice and constructive criticism.

This work was supported by the German Research Foundation (Deut-sche Forschungsgemeinschaft, DFG) through the Institutional Strategy of the University of G\"ottingen, in particular through the research training group GRK 1493 and the Courant Research Center ``Higher Order Structures in Mathematics''. The second author is also grateful for the support received by the Studienstiftung des Deutschen Volkes and the German Academic Exchange Service (DAAD), as part of this collaboration was funded within the framework of a ``DAAD Doktorandenstipendium''.


\end{document}